\documentclass[10pt]{article}

\usepackage[pdftex]{graphicx}
\usepackage{amsmath,amssymb,amsthm}
\usepackage{multirow}
\usepackage{authblk}
\usepackage{cite}
\usepackage{url}
\usepackage{pgf}
\usepackage{color}
\usepackage{todonotes}
\usepackage{booktabs} 
\usepackage{algorithm}
\usepackage{algorithmic}

\usepackage{graphicx}

\newcommand{\mb}[1]{\ensuremath{\boldsymbol{#1}}}
\newcommand{\comment}[1]{}

\newtheorem{theorem}{Theorem}

\newtheorem{proposition}[theorem]{Proposition}
\theoremstyle{definition}

\begin{document}

\title{{Multistage Robust Unit Commitment with \\ Dynamic Uncertainty Sets and Energy Storage}}

\author{\'Alvaro~Lorca~and~Xu~Andy~Sun
\thanks{\'A. Lorca and X. A. Sun are with the H. Milton Stewart School of Industrial and Systems Engineering, Georgia Institute of Technology, Atlanta, GA 30332 USA, e-mail: alvarolorca@gatech.edu, andy.sun@isye.gatech.edu.}}

\date{June 2016}


\maketitle

\begin{abstract}
The deep penetration of wind and solar power is a critical component of the future power grid. However, the intermittency and stochasticity of these renewable resources bring significant challenges to the reliable and economic operation of power systems. Motivated by these challenges, we present a multistage adaptive robust optimization model for the unit commitment (UC) problem, which models the sequential nature of the dispatch process and utilizes a new type of dynamic uncertainty sets to capture the temporal and spatial correlations of wind and solar power. The model also considers the operation of energy storage devices. We propose a simplified and effective affine policy for dispatch decisions, and develop an efficient algorithmic framework using a combination of constraint generation and duality based reformulation with various improvements. Extensive computational experiments show that the proposed method can efficiently solve multistage robust UC problems on the Polish 2736-bus system under high dimensional uncertainty of 60 wind farms and 30 solar farms. The computational results also suggest that the proposed model leads to significant benefits in both costs and reliability over robust models with traditional uncertainty sets as well as deterministic models with reserve rules.
\end{abstract}

Keywords: unit commitment, robust optimization, renewable energy, solar power, wind power, energy storage.

\allowdisplaybreaks

\section{Introduction} \label{Section:Introduction}
The reliable and cost-effective operation of power systems with an abundant presence of wind and solar power depends critically on the competence of optimization methods to effectively manage their uncertainty. The most crucial decision process that faces this challenge is the unit commitment (UC) problem, which schedules generating capacities for the next day and prepares the power system for potentially strong variations in the availability of intermittent renewable resources.

Many efforts within the realm of optimization under uncertainty have been developed for the UC problem. Two main types of methods are stochastic programming and robust optimization. Typically, stochastic programming methods involve scenario trees for modeling uncertain parameters, see for example \cite{CheungWatson2015StochasticUC}, \cite{PapavasiliouOren2015HighPerformance}, \cite{Takriti1996Stochastic}, \cite{WangShahidehpour2008}, and references therein. The stochastic programming framework is versatile, however, it may induce substantial computational difficulties for large-scale problems, and it is difficult to properly represent temporal and spatial correlations within scenario trees. Robust optimization instead relies on the concept of uncertainty set, namely, a deterministic set of realizations of uncertain parameters, which leads to simplified models and improved computational tractability. The research on robust UC is growing rapidly, starting with a static robust model for the contingency constrained UC proposed in \cite{Street2011} and the two-stage robust UC models first developed in \cite{BertsimasSun2013Adaptive}, \cite{Jiang2012}, and \cite{Zeng2012robustUC}. In \cite{BertsimasSun2013Adaptive} a security constrained robust UC model is developed under net load uncertainty, \cite{Jiang2012} presents a robust UC model with wind uncertainty and pumped-storage units, and \cite{Zeng2012robustUC} develops a robust UC model with wind uncertainty and demand response. Other two-stage robust UC approaches have considered generator and transmission line contingencies \cite{WangWatson2013}, the combination with stochastic UC \cite{ZhaoGuan2013}, a worst-case regret objective \cite{Jiang2013MinimaxRegret}, and the use of exact and heuristic approaches to solve bilinear subproblems \cite{Jiang2014TwoStage}, \cite{SunLorca2014}, to name a few works. In two-stage robust UC models, generator on/off decisions are selected in the first stage and then dispatch decisions are made in the second stage with the full knowledge of all the uncertain parameters in the future. This assumption on the knowledge of future uncertainty in the two-stage model is unrealistic, because in the real-life dispatch process the operators only know realizations of uncertainty up to the time of the dispatch decision.


A more accurate UC model should restrict operators' actions to only depend on uncertain parameters realized up to the current decision period. That is, the \emph{non-anticipativity} of dispatch decisions needs to be enforced. Such a robust model is called a \emph{multistage} robust UC model. The benefit of such a model over two-stage robust and deterministic UC models is that the multistage model properly prepares the system's ramping capability to meet future demand variations, which is especially important for systems with limited ramping capacity and significant renewable variations \cite{LorcaSunLitvinovZheng2016AffineUC}.


Multistage robust optimization in its most general form is computationally intractable, but the concept of affine policy has been proposed as an effective approximation, where recourse decisions take the form of an affine function of uncertain parameters \cite{BenTal2004}. In order to handle the large scale of the UC problem, simplified affine policies were proposed in \cite{LorcaSunLitvinovZheng2016AffineUC}, as well as a solution method based on constraint generation. Another application of affine policies for UC can be found in a stochastic UC model in \cite{Warrington2015RollingUC}, where the classic approach of duality reformulation is used and computational results are shown for a 2-dimensional uncertainty set and a 118-bus system. In the context of power systems, the affine policy approach has also been applied in a stochastic economic dispatch problem with energy storage decisions \cite{Warrington2013PolicyBasedReserves}, in a robust optimal power flow problem \cite{Jabr2013AdjustableOPF}, and in a chance-constrained optimal power flow problem \cite{Bienstock2014ChanceConstrained}.


A key component of any robust optimization model is the uncertainty set used to represent uncertain parameters \cite{Delage2015Tutorial}. A particularly important challenge is to capture the correlational structure of uncertain parameters. Most of the existing literature, including the above robust optimization references, have considered \emph{static} uncertainty sets where temporal and spatial correlations are not systematically represented or simply ignored. However, some important efforts have been undertaken to try to improve these uncertainty models. In \cite{Chen2007ARobust}, primitive uncertainties with potential asymmetric distributions are considered as underlying factors that determine the uncertain parameters of interest, which can be used to capture dependencies. In \cite{Minoux2014TwoStage}, state-space representable uncertainty sets are considered, which can also be used to capture certain dependencies. In \cite{LorcaSun2015Wind}, the idea of dynamic uncertainty sets is proposed to capture temporal and spatial correlations in wind speeds.

The present paper significantly improves existing work, especially our previous results in \cite{LorcaSunLitvinovZheng2016AffineUC}, in several important directions. In particular, the paper addresses the following research questions that have not been successfully resolved in the existing literature: (i) How to incorporate temporal and spatial dynamics of uncertainty from both wind and solar generation in the multistage robust UC model? (ii) How to efficiently solve the resulting multistage robust UC models for large-scale systems with high dimensional uncertainty? (iii) How to effectively utilize the dispatch policy obtained from the multistage robust UC model in real-time dispatch? (iv) What is the impact of using energy storage in the multistage robust UC model? The contributions of this paper are summarized as follows:
\begin{enumerate}
  \item We propose a new multistage robust UC model with both wind and solar power uncertainty and energy storage, using a simple and effective affine policy. In this model wind and solar power are dispatchable and their availability is the uncertain component.
  \item We also formulate a new robust look-ahead economic dispatch model that utilizes the affine dispatch policy obtained from the multistage robust UC to improve the robustness of real-time operation.
  \item We develop a data-driven approach to construct dynamic uncertainty sets for capturing joint temporal and spatial correlations of multiple wind and solar farms, including a critical enhancement to reduce the dimensionality of these sets. These sets directly model renewable power, rather than rely on explanatory factors such as wind speed or solar irradiance.
  \item We develop an efficient solution method combining constraint generation and the duality based approach with various algorithmic enhancements, including the use of outer approximation techniques for reformulating inter-temporal constraints, the use of a one-tree Benders implementation, and constraint screening speed-up techniques. The proposed algorithm can solve multistage robust UC models for the Polish 2736-bus power system with high-dimensional uncertainty within a couple of hours on a modest personal computer, which offers a practical solution for real-world operations.
  \item Extensive computational experiments on a simulation platform that mimics the hour-to-hour operation of a real-world power system are carried out to show the benefits of the proposed models and algorithms in comparison to other robust and deterministic approaches.
\end{enumerate}

The remainder of the paper is organized as follows. Section \ref{Section:Model} formulates the multistage robust UC model and the dispatch policy, as well as an economic dispatch (ED) method that exploits the robust adaptive dispatch policy. Section \ref{Section:UncertaintySet} proposes dynamic uncertainty sets for modeling wind and solar power. Section \ref{Section:SolutionMethod} develops the solution algorithms. Section \ref{Section:ComputationalExperiments} presents computational experiments. Section \ref{Section:Conclusion} concludes.

\section{Multistage Robust Unit Commitment Model} \label{Section:Model}
\subsection{Fully-Adaptive Model} \label{Section:Model:FullyAdaptiveModel}
We describe here the fully-adaptive multistage robust UC model considered in this paper. The decisions in this model are composed of binary on/off commitment decisions of conventional generators, dispatch decisions of conventional generators and renewable units, and charge/discharge decisions of storage units. The uncertainty considered in this model corresponds to the power availability at renewable units. The goal of the problem is to determine on/off decisions that are decided before the uncertainty is revealed, and a dispatch policy that determines how all other decisions adapt to uncertainty as it is revealed throughout time, with the objective of minimizing the total worst-case cost. The notation is as follows. Let $\mathcal{N}^d, \mathcal{N}^g, \mathcal{N}^l, \mathcal{N}^r, \mathcal{N}^s, \mathcal{T}$ denote the sets of demand nodes, generators, transmission lines, renewable units (wind or solar), storage units, and time periods, respectively. For the different types of units the corresponding notation is described below.
\begin{enumerate}
	\item For \emph{conventional} generator $i$ at time $t$: $x^o_{it}$, $x^+_{it}$ and $x^-_{it}$ are the on/off, start-up, and shut-down decisions, jointly denoted by vector $\mb{x}$; $p^g_{it}(\overline{\mb{p}}^r_{[t]})$ is the adaptive dispatch \emph{policy} of generator output, which is a function of the uncertain available renewable power realized up to time $t$; $\mb{c}$ is a vector encompassing no-load, start-up, and shut-down costs of generators, and $C^g_i$ is the variable cost of generator $i$; $\underline{p}^g_{it}$ and $\overline{p}^g_{it}$ are the minimum and maximum output levels of generator $i$; $RD_{it}$ and $RU_{it}$ are the ramp-down and ramp-up rates of generator $i$, and $SD_{it}$ and $SU_{it}$ are the ramp rates when generator $i$ shuts down or turns on, respectively.
	\item For \emph{renewable} unit $i$ at time $t$: $\overline{p}^r_{it}$ denotes its available power output, which is an uncertain parameter in the robust UC model \eqref{eq:FullyAdaptiveModel};  $\overline{\mathcal{P}}^r$ is the uncertainty set for the uncertain vector $\overline{\mb{p}}^r$ of available renewable output; $p^r_{it}(\cdot)$ is the dispatch policy of renewable unit $i$.
	\item For \emph{storage} unit $i$ at time $t$: $p^{s+}_{it}(\cdot)$ and $p^{s-}_{it}(\cdot)$ are respectively its discharge and charge policies; $\underline{p}^{s+}_{it}$, $\overline{p}^{s+}_{it}$, $\underline{p}^{s-}_{it}$, and $\overline{p}^{s-}_{it}$ are the limits for its power output and input; and $q^{s}_{i0}$, $\overline{q}^{s}_{i}$, and $\eta^s_i$ are its initial storage level, capacity, and efficiency.
\end{enumerate}
Based on the description above, vector $\mb{x}$ encompasses all binary decisions related to generator commitment, and $\mb{p}(\cdot) = (\mb{p}^g(\cdot),\mb{p}^r(\cdot),\mb{p}^s(\cdot))$ corresponds to a dispatch policy of conventional, renewable, and storage units that are functions of the available renewable power $\overline{\mb{p}}^r$. For the transmission system, let $\alpha^d_{lj}$, $\alpha^g_{li}$, $\alpha^r_{li}$, and $\alpha^s_{li}$ denote the corresponding shift factor values for demands, generators, renewable units, and storage units, respectively, for transmission line $l$, based on a DC power flow model, and let $f_l^{max}$ denote the flow limit for line $l$. Finally, let $d_j^t$ denote the demand level of node $j$ at time $t$.

With the above notation, the fully adaptive robust UC model is formulated as follows:
\begin{subequations} \label{eq:FullyAdaptiveModel}
\begin{align}
& \min\limits_{\mb{x},\mb{p}(\cdot)} \;\; \left\{ \mb{c}^\top x \; + \; \max\limits_{\overline{\mb{p}}^r \in \overline{\mathcal{P}}^r} \; \sum_{i \in \mathcal{N}_g} \sum_{t \in \mathcal{T}} C^g_i \, p^g_{it}(\overline{\mb{p}}^r_{[t]}) \right\} \label{eq:FullyAdaptiveModel:Objective} \\
\mbox{s.t.} \;\; & \mb{x} \in X \label{eq:FullyAdaptiveModel:OnOffConstraints} \\
& x^o_{it} \, \underline{p}^g_{it} \leq p^g_{it}(\overline{\mb{p}}^r_{[t]}) \leq x^o_{it} \, \overline{p}^g_{it} \quad \forall \overline{\mb{p}}^r \in \overline{\mathcal{P}}^r, i \in \mathcal{N}^g, t \in \mathcal{T} \label{eq:FullyAdaptiveModel:GeneratorOutputLimits} \\
& -RD_{it} \, x^o_{it} - SD_{it} \, x^-_{it} \leq p^g_{it}(\overline{\mb{p}}^r_{[t]}) - p^g_{i,t-1}(\overline{\mb{p}}^r_{[t-1]}) \leq RU_{it} \, x^o_{i,t-1} + SU_{it} \, x^+_{it} \label{eq:FullyAdaptiveModel:GeneratorRampingLimits} \\
& \hspace{9cm} \forall \overline{\mb{p}}^r \in \overline{\mathcal{P}}^r, i \in \mathcal{N}^g, t \in \mathcal{T} \notag \\
& 0 \leq p^r_{it}(\overline{\mb{p}}^r_{[t]}) \leq \overline{p}^r_{it} \quad \forall \overline{\mb{p}}^r \in \overline{\mathcal{P}}^r, i \in \mathcal{N}^r, \, t \in \mathcal{T} \label{eq:FullyAdaptiveModel:RenewableUnitOutputLimits} \\
& \underline{p}^{s+}_{it} \leq p^{s+}_{it}(\overline{\mb{p}}^r_{[t]}) \leq \overline{p}^{s+}_{it} \quad \forall \overline{\mb{p}}^r \in \overline{\mathcal{P}}^r, i \in \mathcal{N}^s, \, t \in \mathcal{T} \label{eq:FullyAdaptiveModel:StorageUnitOutputLimits} \\
& \underline{p}^{s-}_{it} \leq p^{s-}_{it}(\overline{\mb{p}}^r_{[t]}) \leq \overline{p}^{s-}_{it} \quad \forall \overline{\mb{p}}^r \in \overline{\mathcal{P}}^r, i \in \mathcal{N}^s, \, t \in \mathcal{T} \label{eq:FullyAdaptiveModel:StorageUnitInputLimits} \\
& 0 \leq q^{s}_{i0} + \sum_{\tau \in [1:t]} \left( \eta^s_i \, p^{s-}_{i \tau}(\overline{\mb{p}}^r_{[\tau]}) - p^{s+}_{i \tau}(\overline{\mb{p}}^r_{[\tau]}) \right) \leq \overline{q}^{s}_{i} \quad \forall \overline{\mb{p}}^r \in \overline{\mathcal{P}}^r,\, i \in \mathcal{N}^s, \, t \in \mathcal{T} \label{eq:FullyAdaptiveModel:StorageCapacityLimits} \\
& -f_l^{max} \leq \sum_{i \in \mathcal{N}^g} \alpha^g_{li} \, p^g_{it}(\overline{\mb{p}}^r_{[t]}) + \sum_{i \in \mathcal{N}^r} \alpha^r_{li} \, p^r_{it}(\overline{\mb{p}}^r_{[t]}) + \sum_{i \in \mathcal{N}^s} \alpha^s_{li} \, \left( p^{s+}_{it}(\overline{\mb{p}}^r_{[t]}) - p^{s-}_{it}(\overline{\mb{p}}^r_{[t]}) \right) \label{eq:FullyAdaptiveModel:LineCapacities} \\
& \hspace{5cm} - \sum_{j \in \mathcal{N}^d} \alpha^d_{lj} \, d_{jt}\leq f_l^{max} \quad \forall \overline{\mb{p}}^r \in \overline{\mathcal{P}}^r,\, l \in \mathcal{N}^l, \, t \in \mathcal{T} \notag \\
& \sum_{j \in \mathcal{N}^d} d_{jt} = \sum_{i \in \mathcal{N}^g} \hspace{-1mm} p^g_{it}(\overline{\mb{p}}^r_{[t]}) + \sum_{i \in \mathcal{N}^r} \hspace{-1mm} p^r_{it}(\overline{\mb{p}}^r_{[t]}) + \sum_{i \in \mathcal{N}^s} \hspace{-1mm} \left( p^{s+}_{it}(\overline{\mb{p}}^r_{[t]}) - p^{s-}_{it}(\overline{\mb{p}}^r_{[t]}) \right) \label{eq:FullyAdaptiveModel:EnergyBalance} \\
& \hspace{11cm} \forall \overline{\mb{p}}^r \in \overline{\mathcal{P}}^r, \, t \in \mathcal{T}. \notag
\end{align}
\end{subequations}

In this problem, the objective \eqref{eq:FullyAdaptiveModel:Objective} consists of minimizing the sum of commitment costs (including no-load, start-up, and shut-down costs) and the worst-case dispatch cost, which is assumed to be linear for ease of exposition, but can be replaced with a piecewise linear function without changing the structure of the problem. Notice that the second-stage decision in \eqref{eq:FullyAdaptiveModel} is a policy function. The overall $\min-\max$ structure in \eqref{eq:FullyAdaptiveModel:Objective} can be equivalently written in the more familiar form of nested $\min-\max$ over time periods, where the dispatch decision variables are vectors rather than policy functions. See \cite[Chapter 2]{Shapiro2009StochBook} for more discussions. Eq. \eqref{eq:FullyAdaptiveModel:OnOffConstraints} represents all the commitment constraints, including as start-up and shut-down constraints and minimum up and down times in the set $X$ (see e.g. \cite{Ostrowski2012TightUC} for details on the formulation). Eq. \eqref{eq:FullyAdaptiveModel:GeneratorOutputLimits} enforces output limits when generators are on, and zero output when they are off. Eq. \eqref{eq:FullyAdaptiveModel:GeneratorRampingLimits} is the ramping constraints. Eq. \eqref{eq:FullyAdaptiveModel:RenewableUnitOutputLimits} restricts the dispatch level of a renewable unit to be bounded by the available power of that unit.  Eqs. \eqref{eq:FullyAdaptiveModel:StorageUnitOutputLimits}-\eqref{eq:FullyAdaptiveModel:StorageUnitInputLimits} are discharge and charge limit constraints for storage units. Eq. \eqref{eq:FullyAdaptiveModel:StorageCapacityLimits} enforces energy storage capacity bounds for storage units. Eq. \eqref{eq:FullyAdaptiveModel:LineCapacities} describes transmission line flow limit constraints. Eq. \eqref{eq:FullyAdaptiveModel:EnergyBalance} enforces system energy balance. Notice that Eqs. \eqref{eq:FullyAdaptiveModel:GeneratorOutputLimits}-\eqref{eq:FullyAdaptiveModel:EnergyBalance} are robust constraints, namely, they must hold for all $\overline{\mb{p}}^r \in \overline{\mathcal{P}}^r$.

In multistage robust UC model \eqref{eq:FullyAdaptiveModel}, the dispatch decision $\mb{p}_t(\overline{\mb{p}}^r_{[t]})$ at time $t$ only depends on $\overline{\mb{p}}^r_{[t]}$, the realization of uncertain available renewable power \emph{up to} time $t$, where $[t] := \{1,\dots,t\}$. This makes the decision policy non-anticipative. In contrast, a \emph{two-stage} robust UC can be formulated similarly to \eqref{eq:FullyAdaptiveModel} but replacing $\mb{p}_t(\overline{\mb{p}}^r_{[t]})$ by $\mb{p}_t(\overline{\mb{p}}^r)$, thus making the dispatch decision at time $t$ adaptive to the realization of uncertain available renewable power over all time periods, which violates non-anticipativity. A more illuminating way to see the difference is discussed in \cite[Section 3.1]{LorcaSunLitvinovZheng2016AffineUC}, where the multistage robust UC is reformulated as a nested sequence of $T$ stages of $\min$-$\max$ problems, while the two-stage model can be reformulated as a $\min$-$\max$-$\min$ problem.






\vspace{-1mm}
\subsection{Affine Dispatch Policy} \label{Section:Model:AffinePolicy}
Problem \eqref{eq:FullyAdaptiveModel} is quite computationally challenging, which can be seen from the fact that decision $\mb{p}(\cdot)$ is in the infinite dimensional space of functions. Alternatively, the computational difficulty can also be appreciated from the nested reformulations mentioned above.


To make the problem computationally tractable, we restrict our attention to affine policies, also known as linear decision rules \cite{BenTal2004, Kuhn2011primal,Jabr2013AdjustableOPF}. To give some insight on the affine dispatch policy, let us write it as $p_{it}^g(\overline{\mb{p}}^r_{[t]})=\hat{p}_{it}^g + \sum_{j\in\mathcal{N}^r}\sum_{\tau\leq t}\alpha_{itj\tau}(\overline{p}^r_{j\tau}-\hat{p}^r_{j\tau})$, where $\hat{p}^r_{j\tau}$ is the forecast (or nominal) available renewable power, $\hat{p}_{it}^g$ is the dispatch level if the realized available renewable power is equal to the forecast, and $\alpha_{itj\tau}$ is the sensitivity coefficient of dispatch on the deviation between the realized and forecast available renewable power. Note that this affine policy depends on the uncertainty at all buses and all time periods prior to $t$. We call such a policy a full affine policy.


It turns out that, for large-scale power systems, the above full affine policy for problem \eqref{eq:FullyAdaptiveModel} is too computationally difficult to solve \cite{LorcaSunLitvinovZheng2016AffineUC}. To deal with this difficulty, we use the following simplified affine policy that adapts to an aggregation of uncertainty:
\begin{subequations} \label{eq:AffinePolicy}
\begin{align}
& \mbox{Conventional units:}\quad p^g_{it}(\overline{\mb{p}}^r_{[t]}) = w^g_{it} + W^g_{it} \; \sum_{j \in \mathcal{N}^r} \overline{p}^r_{jt} \label{eq:AffinePolicy:Generators} \\
& \mbox{Storage units:} \quad p^{s+}_{it}(\overline{\mb{p}}^r_{[t]}) = w^{s+}_{it} + W^{s+}_{it} \; \sum_{j \in \mathcal{N}^r} \overline{p}^r_{jt} \label{eq:AffinePolicy:StorageUnitsOutput} \\
& \hspace{2.4cm} p^{s-}_{it}(\overline{\mb{p}}^r_{[t]}) = w^{s-}_{it} + W^{s-}_{it} \; \sum_{j \in \mathcal{N}^r} \overline{p}^r_{jt} \label{eq:AffinePolicy:StorageUnitsInput}\\
& \mbox{Renewable units:} \quad p^r_{it}(\overline{\mb{p}}^r_{[t]}) = w^r_{it} + W^r_{t} \, \overline{p}^r_{it}, \;\, 
\label{eq:AffinePolicy:RenewableUnits}
\end{align}
\end{subequations}
where $\mb{w}, \mb{W}$ are the new decision variables. In this type of policy, dispatch decisions of generators and storage units depend linearly on the total available renewable power at a system level, and renewable units depend linearly on their own local available power. As we will show, this simplified affine policy performs surprisingly well for \eqref{eq:FullyAdaptiveModel}.


Problem \eqref{eq:FullyAdaptiveModel} under affine policy \eqref{eq:AffinePolicy} can be written in a compact form as
\begin{subequations} \label{eq:AffineModel}
\begin{align}
& \min\limits_{\mb{x}\in X,\mb{w},\mb{W}} \;\; \left\{ \mb{c}^\top \mb{x} \; + \; \max\limits_{\overline{\mb{p}}^r \in \overline{\mathcal{P}}^r} \; \sum_{t \in \mathcal{T}} \mb{C}^\top_t \hspace{-1mm} \left( \mb{w}_t + \mb{W}_{\hspace{-1mm} t} \, \overline{\mb{p}}^r_{t} \right) \right\} \label{eq:AffineModel:Objective} \\
&\mbox{s.t.}\; \mb{w}_t + \mb{W}_{\hspace{-1mm} t} \, \overline{\mb{p}}^r_{t} \in \Omega_t \left( \mb{x}, \, \mb{w}_{[t-1]} + \mb{W}_{\hspace{-1mm} [t-1]} \, \overline{\mb{p}}^r_{[t-1]}, \, \overline{\mb{p}}^r_{t} \right) \quad \forall\overline{\mb{p}}^r \in \overline{\mathcal{P}}^r, t \in \mathcal{T} \label{eq:AffineModel:DispatchConstraints}
\end{align}
\end{subequations}
where \eqref{eq:AffineModel:DispatchConstraints} represents \eqref{eq:FullyAdaptiveModel:GeneratorOutputLimits}-\eqref{eq:FullyAdaptiveModel:EnergyBalance}. Note that \eqref{eq:AffineModel} is still a large-scale robust optimization problem with mixed-integer variables. In section \ref{Section:SolutionMethod}, we will show that exploiting the structure of \eqref{eq:AffineModel} is crucial to efficiently solving it.


\subsection{Policy-guided look-ahead ED method} \label{Section:Model:EDMethods}
The solution of problem \eqref{eq:AffineModel} not only provides a UC schedule $\mb{x}$, but also a dispatch policy $\mb{p}(\cdot)$ that may be utilized in the real-time hour-to-hour dispatch process. We will show that using this policy can significantly improve the flexibility of real-time dispatch.

Consider the real-time dispatch operation at time $t$. Denote the available renewable power realized up to this time as $\overline{\mb{p}}^{r,realized}_{[t]}$. For future time periods $\tau>t$, the available renewable power is forecasted as $\overline{\mb{p}}^{r,forecast}_{\tau}$. Moreover, denote dispatch decisions of all the units realized up to time $t-1$ as $\mb{p}^{realized}_{[t-1]}$. We propose a new ED model, called the \emph{policy-guided look-ahead} ED:
\begin{subequations} \label{eq:PolicyGuidedLAED}
\begin{align}
& \min\limits_{\hat{\mb{p}}_t,\dots,\hat{\mb{p}}_{t+T'}} \;\; \sum_{\tau = t}^{t+T'}\sum_{i \in \mathcal{N}_g} C^g_i \, \hat{p}^g_{i\tau} \label{eq:PolicyGuidedLAED:Objective} \\
\mbox{s.t.} \;\; & \hat{\mb{p}}_t \in \Omega_t (\mb{p}^{realized}_{[t-1]}, \overline{\mb{p}}^{r,realized}_{t}) \label{eq:PolicyGuidedLAED:pt} \\
& \hat{\mb{p}}_{\tau} \in \Omega_{\tau} (\hat{\mb{p}}_{[\tau-1]}, \overline{\mb{p}}^{r,forecast}_{\tau})\;\; \tau \in [t+1:t+T'] \label{eq:PolicyGuidedLAED:ptau} \\
& -RD_{i,t+1} \, x^o_{i,t+1} - SD_{i,t+1} \, x^-_{i,t+1} \leq p^g_{i,t+1}(\overline{\mb{p}}^r_{[t+1]}) - \hat{p}^g_{it} \leq RU_{i,t+1} \, x^o_{it} + SU_{i,t+1} \, x^+_{i,t+1} \label{eq:PolicyGuidedLAED:GeneratorRobustRamping} \\
& \hspace{8cm} \forall \, \overline{\mb{p}}^r_{[t+1]} \in \overline{\mathcal{P}}^r_{[t+1]} \left( \overline{\mb{p}}^{r,realized}_{[t]} \right), \, i \in \mathcal{N}^g \notag \\
&  \sum_{k=t}^{\tau}\left( \eta^s_i \, p^{s-}_{ik}(\overline{\mb{p}}^{r,forecast}_{[k]}) - p^{s+}_{ik}(\overline{\mb{p}}^{r,forecast}_{[k]})\right) = \sum_{k=t}^{\tau} \left( \eta^s_i \, \hat{p}^{s-}_{ik} - \hat{p}^{s+}_{ik} \right) \quad \forall \, i \in \mathcal{N}^s, \, \tau \in [t+1:t+T']. \label{eq:PolicyGuidedLAED:StorageCapacityDeterminedByPolicy}
\end{align}
\end{subequations}
Here $ \hat{\mb{p}}_t$ is the dispatch decision to be implemented at time $t$; $ \hat{\mb{p}}_{\tau}$ for $\tau>t$ is dispatch decision for future time periods under the forecast condition; $T'$ is the number of look-ahead time periods in the ED model; $\Omega_t(\mb{p}^{realized}_{[t-1]}, \overline{\mb{p}}^{r,realized}_{t})$ in \eqref{eq:PolicyGuidedLAED:pt} represents all deterministic dispatch constraints at time $t$ (that is, a deterministic version of eqs. \eqref{eq:FullyAdaptiveModel:GeneratorOutputLimits}-\eqref{eq:FullyAdaptiveModel:EnergyBalance} for time $t$); $\Omega_{\tau}( \hat{\mb{p}}_{[\tau-1]},\overline{\mb{p}}^{r,forecast}_{\tau})$ in \eqref{eq:PolicyGuidedLAED:ptau}
represents all dispatch constraints at time $\tau>t$ which depends on dispatch decisions in previous periods, $ \hat{\mb{p}}_{[\tau-1]}$, and on forecast available renewable power, $\overline{\mb{p}}^{r,forecast}_{\tau}$.

The key constraints  are \eqref{eq:PolicyGuidedLAED:GeneratorRobustRamping} and \eqref{eq:PolicyGuidedLAED:StorageCapacityDeterminedByPolicy}. Constraint \eqref{eq:PolicyGuidedLAED:GeneratorRobustRamping} enforces \emph{robust} ramping requirement for dispatch from time $t$ to $t+1$, where dispatch at $t+1$ uses the affine policy $p^g_{i,t+1}(\overline{\mb{p}}^r_{[t+1]})$ obtained from the multistage UC model \eqref{eq:AffineModel}. Constraint \eqref{eq:PolicyGuidedLAED:StorageCapacityDeterminedByPolicy} enforces that the energy levels of storage units under $\hat{\mb{p}}$ at time $\tau>t$ match those determined by the affine dispatch policy for forecast available renewable power. $\overline{\mathcal{P}}^r_{[t+1]}(\overline{\mb{p}}^{r,realized}_{[t]})$ is the uncertainty set $\overline{\mathcal{P}}^r$ restricted to the observed available renewable power up to time $t$, and projected up to time period $t+1$. The philosophy of constraints \eqref{eq:PolicyGuidedLAED:GeneratorRobustRamping} and \eqref{eq:PolicyGuidedLAED:StorageCapacityDeterminedByPolicy} is that using the affine dispatch policy $p^g_{i,\tau}(\overline{\mb{p}}^r_{[\tau]})$ and $p^{s}_{i\tau}(\overline{\mb{p}}^{r,forecast}_{[\tau]})$ at future time periods can provide effective guidance to the dispatch decision at time $t$, because these affine policies are obtained from the multistage robust UC model which has a holistic view of uncertainty. We will show that this policy-guided look-ahead ED can better utilize storage devices and has a better capability to hedge against uncertainty than a deterministic look-ahead ED model (see Section \ref{Section:ComputationalExperiments:Simulation}).

\section{Dynamic Uncertainty Set for Wind and Solar Power} \label{Section:UncertaintySet}
Wind and solar power present significant temporal and spatial correlations \cite{Xie2011WindIntegration}. In this section, we propose a new type of dynamic uncertainty sets to capture such correlations.
\vspace{-2mm}
\subsection{Mathematical Formulation} \label{Section:UncertaintySet:MathematicalFormulation}
The dynamic uncertainty set for the available wind and solar power over a time horizon $\mathcal{T}$ is given as
\begin{subequations} \label{eq:UncertaintySet}
\begin{align}
& \overline{\mathcal{P}}^r = \Big\{ \overline{\mb{p}}^r = (\overline{p}_{it}^r)_{i,t} \, : \; \exists \; \mb{u}, \, \mb{v} \;\; \mbox{s.t.} \label{eq:UncertaintySet:def} \\
& \hspace{2cm} \overline{p}^r_{it} = f_{it} + g_{it} \, u_{it} \quad \forall \, i \in \mathcal{N}^r, \, t \in \mathcal{T} \label{eq:UncertaintySet:p} \\
& \hspace{2cm} \mb{u}_{t} = \sum_{l=1}^L \mb{A}^l \, \mb{u}_{t-l} + \mb{B} \, \mb{v}_{t} \quad \forall \, t \in \mathcal{T} \label{eq:UncertaintySet:u} \\
& \hspace{2cm} \| \mb{v}_{t} \| \leq \Gamma \quad \forall \, t \in \mathcal{T} \label{eq:UncertaintySet:v} \\
& \hspace{2cm} \sum_{t \in \mathcal{T}} \| \mb{v}_{t} \| \leq \rho \, \Gamma \, |\mathcal{T}| \label{eq:UncertaintySet:BudgetOverTimePeriods} \\
& \hspace{2cm} 0 \leq \overline{p}^r_{it} \leq \overline{p}^{r,max}_{it} \quad \forall \, i \in \mathcal{N}^r, \, t \in \mathcal{T} \Big\}, \label{eq:UncertaintySet:pBounds}
\end{align}
\end{subequations}
where $\overline{p}^r_{it}$ is available power of renewable unit $i$ at time $t$, $\mb{f}$ and $\mb{g}$ account for deterministic seasonal components, $\| \cdot \|$ is a norm, $\Gamma > 0$ is a size parameter, $\rho \in (0,1]$ determines a ``budget over time periods'', and $\overline{p}^{r,max}_{it}$ determines an upper bound on $\overline{p}^r_{it}$. Here $\mb{v}_{t} \in \mathbb{R}^{N_v}$ with $N_v$ between $1$ and $|\mathcal{N}^r|$.

The key feature of \eqref{eq:UncertaintySet} is that both temporal and spatial correlations between uncertain wind and solar power are captured in \eqref{eq:UncertaintySet:u}, where $\mb{A}^l$ and $\mb{B}$ determine the temporal and spatial correlations of uncertain renewable power. More specifically, $\mb{u}_t$ is the uncertainty in the wind and solar power output after the seasonality pattern ($f_{it},g_{it}$) is filtered out; $\mb{u}_t$ includes both temporal and spatial correlations of these uncertain resources in $\mb{A}^l$ and $\mb{B}$ matrices; $\mb{v}_t$ represents residual uncertainty after temporal and spatial correlations are further removed from $\mb{u}_t$. So $\mb{v}_t$ can be viewed as representing a random vector with uncorrelated components over time and space. The support of $\mb{v}_t$ is described by \eqref{eq:UncertaintySet:v}-\eqref{eq:UncertaintySet:BudgetOverTimePeriods}. The size of the support is controlled by $\Gamma$, which is analogous to the maximum number of standard deviations that we allow for variations in each component of $\mb{v}_t$. We call \eqref{eq:UncertaintySet} a \emph{dynamic} uncertainty set. As a special case, if the dimension of $\mb{v}_t$ is the same as that of $\mb{u}_t$, $\mb{B}$ is the identity matrix, $\mb{A}^l$'s are zero, and $\rho=1$, we obtain a \emph{static} uncertainty set that ignores temporal and spatial correlations and is separable over time periods, similar to the budget uncertainty set used in literature (see e.g. \cite{BertsimasSun2013Adaptive,LorcaSunLitvinovZheng2016AffineUC}).


To further illustrate the meaning of $\Gamma$, consider the static case described above (where $\mb{u}_{t} = \mb{v}_{t}$ and $\rho=1$), the norm is the $\ell_\infty$ norm, i.e. $\| \cdot \| = \| \cdot \|_\infty$, and $g_{it}$ corresponds to the standard deviation of available power at renewable unit $i$ and time $t$. In this case we obtain $f_{it} - \Gamma \, g_{it} \leq \overline{p}^r_{it} \leq f_{it} + \Gamma \, g_{it}$ and no coupling relations between renewable units or time periods, and we can thus interpret $\Gamma$ as the ``number of standard deviations'' that we allow for available power variations at each renewable unit and time period. This is the most basic uncertainty set corresponding to a box. For the general dynamic uncertainty set, the intuition for the choice of $\Gamma$ is similar, except that $\Gamma$ will now determine the ``number of standard deviations'' that we allow for variations of each component in $\mb{v}_{t}$.

The concept of dynamic uncertainty sets was first proposed in \cite{LorcaSun2015Wind}, where uncertainty in \emph{wind speed} is modeled and a power curve is used to transform wind speed into available wind power. In this paper, we directly model uncertainty in renewable power to improve computational efficiency. Furthermore, we model both wind and solar power uncertainty and their correlations. Notice that the norm in \eqref{eq:UncertaintySet:v}, \eqref{eq:UncertaintySet:BudgetOverTimePeriods} can be any norm such as $\ell_1, \ell_2, \ell_\infty$ or a combination thereof, such as the intersection of $\ell_1$ and $\ell_\infty$.

\subsection{Parameter Estimation And Dimension Reduction} \label{Section:UncertaintySet:ParameterEstimation}
To estimate the parameters of the dynamic uncertainty set \eqref{eq:UncertaintySet}, we consider the following stochastic model for available power of renewable units:
\begin{subequations} \label{eq:StochasticModel}
\begin{align}
& \tilde{p}^r_{it} = f_{it} + g_{it} \, \tilde{u}_{it} \quad \forall \, i \in \mathcal{N}^r, \, t \in \mathcal{T} \label{eq:StochasticModel:p} \\
& \tilde{\mb{u}}_{t} = \sum_{l=1}^L \mb{A}^l \, \tilde{\mb{u}}_{t-l} + \tilde{\mb{\epsilon}}_{t} \quad \forall \, t \in \mathcal{T} \label{eq:StochasticModel:u},
\end{align}
\end{subequations}
where $\tilde{p}^r_{it}$ is the power available at renewable unit $i$ at time $t$, and $\tilde{\mb{\epsilon}}_{t}$ is an i.i.d. random vector under certain distribution. In this paper, we assume $\tilde{\mb{\epsilon}}_t$ to have a multivariate normal distribution centered at zero with covariance matrix $\mb{\Sigma}$. Other distributions could be assumed with a proper choice of the norm in \eqref{eq:UncertaintySet:v}, \eqref{eq:UncertaintySet:BudgetOverTimePeriods} in the dynamic uncertainty set (e.g. using the approach in \cite{Chen2007ARobust}).

In order to estimate the parameters of this model, $\mb{f}$ and $\mb{g}$ can be estimated using linear regression after identifying daily seasonality in the volatility of available renewable power. Next, $\tilde{\mb{u}}_{t}$ corresponds to a multivariate autoregressive process, and given a choice of time lag $L$, $\mb{A}^l$ and $\mb{\Sigma}$ can be estimated using statistical inference techniques from time series analysis \cite{Reinsel2003Elements}. Using this, $\mb{B}$ in \eqref{eq:UncertaintySet} can be estimated by Cholesky decomposition of $\mb{\Sigma}$ as $\mb{\Sigma} = \mb{B} \mb{B}^\top$.

At this point, it is important to note that the dimension of the uncertainty set plays a fundamental role in the difficultness of solving the associated robust optimization problem. Under the presence of many renewable units, $\overline{\mathcal{P}}^r$ can have a large dimension. We can reduce the dimension of $\overline{\mathcal{P}}^r$ by principal component analysis as follows. The matrix $\mb{\Sigma}$ can be eigen-decomposed as $\mb{\Sigma} = \mb{V} \mb{\Lambda} \mb{V}^\top$, where $\mb{V}$ contains the eigenvectors and $\mb{\Lambda}$ has the eigenvalues in the diagonal. Then, we ignore the smaller eigenvalues in $\mb{\Lambda}$ and the corresponding eigenvectors in $\mb{V}$ by removing the corresponding columns of $\mb{B} = \mb{V} \mb{\Lambda}^{1/2}$. In this way, the number of columns $N_v$ in \eqref{eq:UncertaintySet} left in $\mb{B}$ can be any number from $1$ to $|\mathcal{N}^r|$. If $N_v$ is selected too close to $|\mathcal{N}^r|$, then a high-dimensional uncertainty set is obtained, resulting in a large problem. If $N_v$ is too close to $1$, then the uncertainty representation may be too inaccurate. The right balance will depend on the particular instance solved.

\section{Solution Method} \label{Section:SolutionMethod}
The affine multistage robust UC model \eqref{eq:AffineModel} is a so-called semi-infinite program, i.e. there are finite number of decision variables, but infinite number of constraints. Due to this, a deterministic counterpart of \eqref{eq:AffineModel} needs to be formulated. There are two main classes of approaches for this purpose. The most widely used approach is the duality based approach \cite{Bental2009robustbook} that replaces each robust constraint by its dual program with additional variables and constraints. The other less explored approach is based on constraint generation \cite{BertsimasDunning2015Reformulations, LorcaSunLitvinovZheng2016AffineUC}, which dynamically generates violated scenarios and the associated deterministic constraints. These two approaches are combined in the solution method proposed in this section, and special structures of \eqref{eq:AffineModel} are exploited.


Section \ref{Section:SolutionMethod:ConstraintGeneration} introduces the basic constraint generation framework. Section \ref{Section:SolutionMethod:SimpleConstraints} exploits the special structure of the robust generation limit and energy balance constraints in \eqref{eq:AffineModel}. Section \ref{Section:SolutionMethod:OuterApproximation} presents an outer approximation method for reformulating the inter-temporal constraints in \eqref{eq:AffineModel}. Section \ref{Section:SolutionMethod:AlgorithmicEnhancements} discusses several techniques to further enhance the efficiency of the overall algorithm.

\subsection{Constraint Generation Framework} \label{Section:SolutionMethod:ConstraintGeneration}
Constraint generation (CG) is recently applied to solve large-scale robust optimization problems \cite{LorcaSunLitvinovZheng2016AffineUC,BertsimasDunning2015Reformulations}. The master robust UC problem in the CG algorithm can be written as
\begin{subequations} \label{eq:MasterProblem}
\begin{align}
& \min\limits_{(\mb{x},\mb{w},\mb{W},z) \in \Omega} \quad \mb{c}^\top \mb{x} \, + \, z \\
& \mbox{s.t.} \;\; \mb{a}_k (\mb{W})^\top \overline{\mb{p}}^r \; \leq \; b_k(\mb{x},\mb{w},z) \;\; \forall \, k \in [K], \, \overline{\mb{p}}^r \in P_k, \label{eq:MPconstr}
\end{align}
\end{subequations}
where $z$ corresponds to the worst-case dispatch cost, $K$ is the number of robust constraints in the problem, and $P_k$ is the current set of extreme points of the uncertainty set $\overline{\mathcal{P}}^r$ considered for the $k$-th robust constraint. The basic CG algorithm solves the master problem \eqref{eq:MasterProblem}, and checks if the $k$-th robust constraint is violated by the current solution, and if so, the associated worst-case scenario from $\overline{\mathcal{P}}^r$ is added to $P_k$, and the master program is solved again. This method is formally presented in Algorithm \ref{Algorithm:ConstraintGeneration}.

\begin{algorithm}
\caption{Constraint generation}
 \label{Algorithm:ConstraintGeneration}
\small
\begin{algorithmic}[1]
\REPEAT
    \STATE{Solve Master Problem \eqref{eq:MasterProblem}}
    \FORALL{$k \in \{1,...,K\}$}
        \STATE{$\overline{\mb{p}}^r_k \leftarrow \mbox{argmax} \left\{ \mb{a}_k (\mb{W})^\top \overline{\mb{p}}^r: \; \overline{\mb{p}}^r \in \overline{\mathcal{P}}^r \right\}$}
        \STATE{If $\mb{a}_k (\mb{W})^\top \overline{\mb{p}}^r_k > b_k(\mb{x},\mb{w},z)$ let $P_k \leftarrow P_k \cup \{\overline{\mb{p}}^r_k\}$}
    \ENDFOR
\UNTIL{$\mb{a}_k (\mb{W})^\top \overline{\mb{p}}^r_k \leq b_k(\mb{x},\mb{w},z)$ for all $k \in [K]$}
\end{algorithmic}
\end{algorithm}


To give a concrete example of the robust constraint \eqref{eq:MPconstr}, let us consider the upper output limit constraint in \eqref{eq:FullyAdaptiveModel:GeneratorOutputLimits} for a generator $i$ at time $t$ under affine policy \eqref{eq:AffinePolicy}. The robust constraint in this case is given by
\begin{align*}
& w^g_{it} + W^g_{it} \; \sum_{j \in \mathcal{N}^r} \overline{p}^r_{jt} \leq x^o_{it} \, \overline{p}^g_{it} \quad \forall \overline{\mb{p}}^r \in \overline{\mathcal{P}}^r,
\end{align*}
where $a_{k j \tau}(\mb{W}) = W^g_{it}$ for $\tau = t$ and $a_{k j \tau}(\mb{W}) = 0$ for $\tau \neq t$, for any $j \in \mathcal{N}^r$, and $b_k(\mb{x},\mb{w},z) = x^o_{it} \, \overline{p}^g_{it} - w^g_{it}$. All other robust constraints in the master problem are similarly defined.

This basic CG framework is the starting point to develop a practical algorithm for solving large-scale robust UC problems. The key is to fully exploit the structure of \eqref{eq:AffineModel}. In the full algorithm described in this section, we will handle transmission line flow limit constraints \eqref{eq:FullyAdaptiveModel:LineCapacities} through CG; for all other constraints, we use more efficient reformulations.


\subsection{Reformulation of generation limit and balance constraints} \label{Section:SolutionMethod:SimpleConstraints}
The deterministic counterparts of the robust generation limit \eqref{eq:FullyAdaptiveModel:GeneratorOutputLimits} and the energy balance constraints \eqref{eq:FullyAdaptiveModel:EnergyBalance} in the affine UC model \eqref{eq:AffineModel} can be explicitly derived without any dualization of the uncertainty set.


\subsubsection{Generation limit constraints} \label{Section:SolutionMethod:SimpleConstraints:OutputLimits}
Due to the structure of the simplified affine policy \eqref{eq:AffinePolicy}, we can directly identify the worst-case scenarios for the robust generation limit constraints.
\begin{proposition} \label{Proposition:OutputLimitReformulation}
Under affine policy \eqref{eq:AffinePolicy}, robust generation limit constraints \eqref{eq:FullyAdaptiveModel:GeneratorOutputLimits} are equivalent to
\begin{align}
& x^o_{it} \, \underline{p}^g_{it} \leq w^g_{it} + W^g_{it} \; \overline{p}^{r,total}_{t} \leq x^o_{it} \, \overline{p}^g_{it} \quad \forall \, i \in \mathcal{N}^g, \, t \in \mathcal{T}, \, \overline{p}^{r,total}_{t} \in \left\{ \overline{p}^{r,total,min}_{t}, \, \overline{p}^{r,total,max}_{t} \right\}, \label{eq:GeneratorOutputLimitsReformulation}
\end{align}
where
\begin{subequations} \label{eq:MinAndMaxTotalAvailableRenewablePower}
\begin{align}
& \overline{p}^{r,total,min}_{t} =  \min\limits_{\overline{\mb{p}}^r \in \overline{\mathcal{P}}^r} \; \sum_{j \in \mathcal{N}^r} \overline{p}^r_{jt} \quad \forall \, \, t \in \mathcal{T} \label{eq:MinAndMaxTotalAvailableRenewablePower:Min} \\
& \overline{p}^{r,total,max}_{t} =  \max\limits_{\overline{\mb{p}}^r \in \overline{\mathcal{P}}^r} \; \sum_{j \in \mathcal{N}^r} \overline{p}^r_{jt} \quad \forall \, \, t \in \mathcal{T}. \label{eq:MinAndMaxTotalAvailableRenewablePower:Max}
\end{align}
\end{subequations}
\end{proposition}
\begin{proof}[Proof of Proposition \ref{Proposition:OutputLimitReformulation}.]
The result follows from the fact that
\[ \hspace{-0.1cm} \max\limits_{\overline{\mb{p}}^r \in \overline{\mathcal{P}}^r} a \hspace{-0.1cm} \sum_{j \in \mathcal{N}^r} \hspace{-0.1cm} \overline{p}^r_{jt} = \max \left\{ a \, \overline{p}^{r,total,min}_{t}, \, a \, \overline{p}^{r,total,max}_{t} \right\}, \]
for any given $a$.
\end{proof}
A similar result holds for robust constraints of storage units' output and input limits \eqref{eq:FullyAdaptiveModel:StorageUnitOutputLimits} and \eqref{eq:FullyAdaptiveModel:StorageUnitInputLimits}, as well as for renewable unit output limit constraints \eqref{eq:FullyAdaptiveModel:RenewableUnitOutputLimits}. Details are omitted for space. 

\subsubsection{Energy balance constraints} \label{Section:SolutionMethod:SimpleConstraints:EnergyBalance}
The deterministic counterpart of the robust energy balance constraints can also be obtained in closed form.
\begin{proposition} \label{Proposition:EnergyBalanceReformulation}
Under affine policy \eqref{eq:AffinePolicy}, robust energy balance constraints \eqref{eq:FullyAdaptiveModel:EnergyBalance} are equivalent to the following system of equations, for every $t\in\mathcal{T}$,
\begin{subequations} \label{eq:EnergyBalanceReformulation}
\begin{align}
&\hspace{-2mm} \sum_{i \in \mathcal{N}^g} w^g_{it} + \sum_{i \in \mathcal{N}^r} w^r_{it} + \sum_{i \in \mathcal{N}^s} \left( w^{s+}_{it} - w^{s-}_{it} \right) = \sum_{j \in \mathcal{N}^d} d_{jt}  \label{eq:EnergyBalanceReformulation:w} \\
& \hspace{-2mm} \sum_{i \in \mathcal{N}^g} W^g_{it} + W^r_{t} + \sum_{i \in \mathcal{N}^s} (W^{s+}_{it}-W^{s-}_{it}) = 0, \label{eq:EnergyBalanceReformulation:W}
\end{align}
\end{subequations}
whenever the uncertainty set $\overline{\mathcal{P}}^r$ is full-dimensional.
\end{proposition}
\begin{proof}[Proof of Proposition \ref{Proposition:EnergyBalanceReformulation}.]
The robust energy balance constraint for each $t$ can be written compactly as $\mb{a}_t(\mb{W})^\top\overline{\mb{p}}_t^r = {b}_t(\mb{w})$ for all $\overline{\mb{p}}_t^r \in \overline{\mathcal{P}}^r$, where $\mb{a}_t(\mb{W})$ and $b_t(\mb{w})$ are linear in $\mb{W}$ and $\mb{w}$, respectively. If $\overline{\mathcal{P}}^r$ is full-dimensional, then $\mb{a}_t(\mb{W})=\mb{0}$ and ${b}_t(\mb{w})=0$ must hold, which gives \eqref{eq:EnergyBalanceReformulation:w}-\eqref{eq:EnergyBalanceReformulation:W}.
\end{proof}

\subsection{Outer approximation for inter-temporal constraints} \label{Section:SolutionMethod:OuterApproximation}
The worst-case cost constraint, ramping constraints \eqref{eq:FullyAdaptiveModel:GeneratorRampingLimits}, and the storage capacity constraints \eqref{eq:FullyAdaptiveModel:StorageCapacityLimits} all involve decisions over consecutive time periods, i.e. they induce inter-temporal coupling between dispatch decisions. Dualizing these robust constraints introduces a large number of new variables and constraints, while directly applying CG may lead to slow convergence. In this section, we introduce an outer approximation method for efficient reformulation. Observe that, due to the simplified affine policy structure \eqref{eq:AffinePolicy}, the inter-temporal robust constraints in \eqref{eq:AffineModel} only depend on total system-level available renewable power rather than on bus-level details. So we can project the bus-level uncertainty set \eqref{eq:UncertaintySet} to its equivalence for system-level uncertainty. However, this latter uncertainty set still involves a large number of variables. Thus, we use outer approximation (OA) to further reduce its dimension. This technique is general. We use it here to reformulate inter-temporal constraints, and in section \ref{Section:SolutionMethod:AlgorithmicEnhancements:UpperBounds} for screening transmission constraints.

The inter-temporal robust constraints in \eqref{eq:AffineModel} can be written as
\begin{align}
& \max\limits_{\overline{\mb{p}}^{r,total}_{[t_1:t_2]} \in \overline{\mathcal{P}}^{r,total}_{[t_1:t_2]}} \; \sum_{t=t_1}^{t_2} a^{total}_t (\mb{W}) \; \overline{p}^{r,total}_t \; \leq \; b(\mb{x},\mb{w},z), \label{eq:TimeCoupledRobustConstraint}
\end{align}
where
\begin{align*}
& \overline{\mathcal{P}}^{r,total}_{[t_1:t_2]} = \Big\{ \overline{\mb{p}}^{r,total}_{[t_1:t_2]}: \; \exists \, \overline{\mb{p}}^r \in \overline{\mathcal{P}}^r \mbox{ s.t.} \;\; \overline{p}^{r,total}_t = \sum_{j \in \mathcal{N}^r} \overline{p}^r_{jt} \;\; \forall t \in [t_1:t_2] \Big\}
\end{align*}
is the projection of bus-level uncertainty set $\overline{\mathcal{P}}^{r}$ unto the total available renewable power, and $a^{total}_t (\mb{W}), b(\mb{x},\mb{w},z)$ are properly defined depending on the particular robust constraint. We then replace $\overline{\mathcal{P}}^{r,total}_{[t_1:t_2]}$ by the following OA $\widehat{\mathcal{P}}^{r,total}_{[t_1:t_2]}$.
\begin{align}
& \widehat{\mathcal{P}}^{r,total}_{[t_1:t_2]} = \Big\{ \; \overline{\mb{p}}^{r,total}_{[t_1:t_2]}: \label{eq:OA}\\
& \hspace{2cm} \overline{p}^{r,total,min}_{t} \leq \overline{p}^{r,total}_t \leq \overline{p}^{r,total,max}_{t} \;\; \forall t \in [t_1:t_2] \notag\\
& \hspace{2cm} \underline{\triangle}^{total}_{t} \leq \overline{p}^{r,total}_t - \overline{p}^{r,total}_{t-1} \leq \overline{\triangle}^{total}_{t} \;\; \forall t \in [t_1+1:t_2] \; \Big\},\notag
\end{align}
where
\begin{subequations} \label{eq:MinAndMaxDeltaOfTotalAvailableRenewablePower}
\begin{align}
&\hspace{-2mm} \underline{\triangle}^{total}_{t} =  \min\limits_{\overline{\mb{p}}^r \in \overline{\mathcal{P}}^r} \sum_{j \in \mathcal{N}^r} \left( \overline{p}^r_{jt} - \overline{p}^r_{j,t-1} \right) \; \forall \, t \in [t_1+1:t_2] \\ 
& \hspace{-2mm} \overline{\triangle}^{total}_{t} =  \max\limits_{\overline{\mb{p}}^r \in \overline{\mathcal{P}}^r} \sum_{j \in \mathcal{N}^r} \left( \overline{p}^r_{jt} - \overline{p}^r_{j,t-1} \right) \; \forall \, t \in [t_1+1:t_2]. 
\end{align}
\end{subequations}
Observe that this set is indeed an OA of $\overline{\mathcal{P}}^{r,total}_{[t_1:t_2]}$ (i.e. $\overline{\mathcal{P}}^{r,total}_{[t_1:t_2]} \subset \widehat{\mathcal{P}}^{r,total}_{[t_1:t_2]}$) due to \eqref{eq:MinAndMaxDeltaOfTotalAvailableRenewablePower}, thus ensuring robust feasibility when replacing $\overline{\mathcal{P}}^{r,total}_{[t_1:t_2]}$ in \eqref{eq:TimeCoupledRobustConstraint}. Also observe that the OA \eqref{eq:OA} only involves the $\overline{p}^{r,total}_t$ variables, whereas $\overline{\mathcal{P}}^{r,total}_{[t_1:t_2]}$ has many more additional variables $\mb{u},\mb{v}$ as in \eqref{eq:UncertaintySet}. Thus, solving \eqref{eq:TimeCoupledRobustConstraint} over the OA set becomes much faster than over the original uncertainty set $\overline{\mathcal{P}}^{r,total}_{[t_1:t_2]}$. The proposition below summarizes this result, and the use of duality based approach to reformulate the resulting robust constraint.
\begin{proposition} \label{Proposition:OuterApproximation}
Robust constraint \eqref{eq:TimeCoupledRobustConstraint} is implied by
\begin{align*}
& \max\limits_{\overline{\mb{p}}^{r,total}_{[t_1:t_2]} \in \widehat{\mathcal{P}}^{r,total}_{[t_1:t_2]}} \; \sum_{t=t_1}^{t_2} a^{total}_t (\mb{W}) \; \overline{p}^{r,total}_t \; \leq \; b(\mb{x},\mb{w},z),
\end{align*}
which is equivalent to the existence of vectors $\overline{\mb{\pi}},\underline{\mb{\pi}},\overline{\mb{\phi}},\underline{\mb{\phi}} \geq 0$ such that
\begin{align*}
& \sum_{t=t_1}^{t_2} \left( \overline{p}^{r,total,max}_{t} \, \overline{\pi}_t - \overline{p}^{r,total,min}_{t} \, \underline{\pi}_t \right) +  \sum_{t=t_1+1}^{t_2} \left( \overline{\triangle}^{total}_{t} \, \overline{\phi}_t - \underline{\triangle}^{total}_{t} \, \underline{\phi}_t \right) \;\; \leq \;\; b(\mb{x},\mb{w},z) \\
&  \overline{\pi}_t - \underline{\pi}_t + \overline{\phi}_t - \underline{\phi}_t - \overline{\phi}_{t+1} + \underline{\phi}_{t+1} = a^{total}_t (\mb{W}) \quad \forall t \in [t_1:t_2] \\
&  \overline{\phi}_{t_1} = \underline{\phi}_{t_1} = \overline{\phi}_{t_2+1} + \underline{\phi}_{t_2+1} = 0.
\end{align*}
\end{proposition}
\begin{proof}[Proof of Proposition \ref{Proposition:OuterApproximation}.]
The result follows from $\overline{\mathcal{P}}^{r,total}_{[t_1:t_2]} \subset \widehat{\mathcal{P}}^{r,total}_{[t_1:t_2]}$ and then taking the dual over $\widehat{\mathcal{P}}^{r,total}_{[t_1:t_2]}$.
\end{proof}


Finally, we note that in \cite{LorcaSunLitvinovZheng2016AffineUC} the separability over time periods of a static uncertainty set is exploited to reformulate ramping constraints in a simple way, also exploiting the fact that these constraints only couple two time periods. However, the more general dynamic uncertainty sets \eqref{eq:UncertaintySet} are not separable over time periods, and energy storage constraints couple up to $T$ time periods. Due to this, the technique presented above based on OA and the duality based approach is a critical enhancement that allows efficiently handling energy storage and dynamic uncertainty sets in the multistage robust UC.

See \cite[Chapter 6]{Tuybook} for more discussion on outer approximations.

\subsection{Further Algorithmic Enhancements} \label{Section:SolutionMethod:AlgorithmicEnhancements}

The CG framework is further enhanced by the following techniques.

\subsubsection{One-tree Benders implementation} \label{Section:SolutionMethod:AlgorithmicEnhancements:OneTreeBenders}
The presence of binary variables in the master problem implies that if the constraint generation approach presented in \ref{Section:SolutionMethod:ConstraintGeneration} is directly used, then a potentially large number of mixed-integer programs will have to be solved throughout the algorithm, which may be quite slow. We propose an efficient alternative which is to use the one-tree Benders approach, in which the solver builds only one branch-and-bound tree and adds the generated constraints as the branch-and-bound process unfolds. Solver callbacks are required in this implementation. For example, lazy constraint callbacks \cite{Cplex2009} are required in CPLEX.

Another alternative, discussed in \cite{LorcaSunLitvinovZheng2016AffineUC}, consists of ``fixing and releasing'' binary variables in order to generate multiple cuts to the mixed-integer master problem. However, this still requires potentially building from scratch more than one branch-and-bound tree and is dominated by the more efficient one-tree Benders approach.

\subsubsection{Constraint screening using fast computed upper bounds} \label{Section:SolutionMethod:AlgorithmicEnhancements:UpperBounds}
Each iteration of the CG algorithm needs to solve the following separation problem for each robust constraint $k$, by checking whether
\begin{align}
& \max\limits_{\overline{\mb{p}}^r \in \overline{\mathcal{P}}^r} \; \mb{a}_k (\mb{W})^\top \overline{\mb{p}}^r \; \leq \; b_k(\mb{x},\mb{w},z) \label{eq:Subproblem}
\end{align}
holds for the fixed $z,\mb{x},\mb{w},\mb{W}$. It amounts to solving a linear program over the dynamic uncertainty set \eqref{eq:UncertaintySet} for each $k$. This can be time consuming for large instances. We propose to screen the robust constraints in the following way. In each iteration of the CG master problem, calculate an \emph{upper bound} $ub_k(\mb{W})$ for the left-hand side of \eqref{eq:Subproblem} as
\begin{align}\label{eq:computeUpperBound}
ub_k(\mb{W}) = \max_{\overline{\mb{p}}^r \in \widehat{\mathcal{P}}^r_k} \; \mb{a}_k (\mb{W})^\top \overline{\mb{p}}^r,
\end{align}
where $\widehat{\mathcal{P}}^r_k$ is an outer approximation of $\overline{\mathcal{P}}^r$, so that $\max_{\overline{\mb{p}}^r \in \overline{\mathcal{P}}^r} \; \mb{a}_k (\mb{W})^\top \overline{\mb{p}}^r \; \leq \; ub_k(\mb{W})$. Then, before solving \eqref{eq:Subproblem} for each robust constraint $k$, we check whether $ub_k(\mb{W}) \leq b_k(\mb{w},z)$. If this holds, then \eqref{eq:Subproblem} must also hold, and we do not need to solve a linear program over the uncertainty set. Otherwise, we solve the linear program to check if \eqref{eq:Subproblem} holds or not.

%


In order for this screening process to be efficient, \eqref{eq:computeUpperBound} must be solved very fast. For this purpose, we construct interval type sets $\widehat{\mathcal{P}}^r_k$ for robust transmission constraints, under which upper bounds can be computed by simply checking the sign of the elements in $\mb{W}$. Since typically many robust constraints are far from being violated (e.g., some transmission lines are rarely congested), such robust constraints will be screened very rapidly using this technique.

\subsubsection{Strategy for checking loose constraints} \label{Section:SolutionMethod:AlgorithmicEnhancements:Strategies}
After checking for the feasibility of all robust constraints in the problem, it is possible that several of them are quite far from being violated. These loose constraints are unlikely to become violated in the next iteration. Given this, we restrict the set of robust constraints to those that were violated or close to being violated in the last iteration, and only get back to checking all robust constraints once the master problem has converged. This process can be repeated as needed until the optimal solution has been found and feasibility is ensured for all robust constraints.

\section{Computational Experiments} \label{Section:ComputationalExperiments}
We conduct extensive computational experiments to evaluate the solution method and to understand the benefits of the new model. The experiments are carried out using an adapted version of the 2736-bus Polish system \cite{Zimmerman11MATPOWER}. The system contains 289 generators (28880 MW of total capacity), 60 wind farms (10689 MW installed), 30 solar farms (6299 MW installed), 10 storage units (600 MW of total output capacity), 2011 demand nodes (17831 MW average, 22594 MW peak) and 100 transmission lines. The energy storage capacity of each storage unit corresponds to five hours at full output, their ramping capacities are unconstrained, and their efficiency is 80\% \cite{Denholm2013ValueOfStorage}. We use wind and solar power data from NREL's Western Wind and Solar Integration Datasets \cite{Potter2008CreatingDataset}. 
All the experiments have been implemented using Python 2.7 in a PC with an Intel Core i5 processor at 2.4 GHz with 4GB memory, using CPLEX 12.6 as the MIP solver. Section \ref{Section:ComputationalExperiments:SolutionMethodPerformance} studies the performance of the proposed solution method, including a comparison of solution quality obtained with and without the outer-approximation technique. Section \ref{Section:ComputationalExperiments:Simulation} evaluates the advantages of the proposed approach.


\subsection{Performance of the solution method} \label{Section:ComputationalExperiments:SolutionMethodPerformance}
Here we study how the various techniques presented in section \ref{Section:SolutionMethod} contribute to an effective solution method. In the experiments, we use a horizon of $T=24$ hours with hourly intervals. The $f,g,\mb{A},\mb{B}$ parameters of the dynamic uncertainty set \eqref{eq:UncertaintySet} are estimated using 30 days of the NREL data. The time lag is set as $L=1$, and the dimension of $\mb{v}$ is set as $N_v = 25$, which offers a good balance between uncertainty representability and computational tractability. For the size parameters, we set $\rho = 0.1$ and test various $\Gamma$ values. The norm in \eqref{eq:UncertaintySet:u}-\eqref{eq:UncertaintySet:v} is defined as $\| \cdot \|_{l_1 \cap l_\infty} = \max \left\{ \| \cdot \|_1 / \sqrt{N_v} \; , \;\; \| \cdot \|_\infty \right\}$,
resulting in a polyhedral uncertainty set. An optimality gap of 1\% is used for solving MIPs in the CG algorithm.
\vspace{-2mm}
\begin{table}[!h] 
	\renewcommand{\arraystretch}{1.1}
	\caption{Solution time (hours) for the 2736-bus system} 
	\label{Table:SolutionTimes:2736BusSystem}
	\centering
	\begin{tabular}{ccccccc}
		\toprule
		$\Gamma$       & 0.25 & 0.5  & 1    & 2    & 3    & 4 \\
		\hline
		CG            & T    & T    & T    & T    & T    & T     \\
		CG + OTB           & T    & T    & T    & T    & T    & T     \\
		CG + OTB + OA      & 1.53 & 2.16 & 1.50 & 1.79 & 1.59 & 2.64  \\
		CG + OTB + OA + CS & 0.96 & 2.04 & 1.27 & 1.45 & 1.30 & 0.79  \\
		\bottomrule
	\end{tabular}
\end{table}

Table \ref{Table:SolutionTimes:2736BusSystem} presents the computation time of different combinations of solution techniques presented in Section \ref{Section:SolutionMethod}. In this Table, ``CG'' corresponds to the basic constraint generation algorithm described in Section \ref{Section:SolutionMethod:ConstraintGeneration}, ``OTB'' is the one-tree Benders implementation discussed in Section \ref{Section:SolutionMethod:AlgorithmicEnhancements:OneTreeBenders}, ``OA'' uses the outer-approximation technique in Section \ref{Section:SolutionMethod:OuterApproximation}, and ``CS'' uses constraint screening and the strategy for checking loose constraints in Sections \ref{Section:SolutionMethod:AlgorithmicEnhancements:UpperBounds} and \ref{Section:SolutionMethod:AlgorithmicEnhancements:Strategies}. All the methods incorporate the reformulations in Section \ref{Section:SolutionMethod:SimpleConstraints}, so ``CG+OTB+OA+SC'' is the full solution method presented in Section \ref{Section:SolutionMethod}. ``T'' stands for reaching a time limit of 6 hours.

We can observe that all of the enhancements are important for efficiency. The basic CG has to solve many difficult mixed integer programs, making it very slow for such a large-scale instance. Method ``OTB'' builds only one branch-and-bound tree and generates the constraints in an integrated way, 
however, the number of robust constraints generated through CG is still very large, making the process slow. This is fixed in ``CG+OTB+OA'' by reformulating the robust inter-temporal constraints for worst-case cost, ramping, and storage capacities, in such a way that they are enforced throughout the whole progress of the algorithm with a simple computational representation, leaving the sequential generation of constraints only for transmission. Finally, ``CG+OTB+OA+SC'' further improves the algorithm by reducing the overall number of separation problems \eqref{eq:Subproblem} solved, through quickly recognizing several robust constraints that are not violated.

To further show the effectiveness of outer approximation, we study its tightness. In particular, Table \ref{Table:WorstCaseCost:2736BusSystem} compares the worst-case cost obtained using the outer approximation method in section \ref{Section:SolutionMethod:OuterApproximation} to that obtained without using this technique (with running time longer than 6 hours). We can observe that the loss of solution quality is small, specially for small values of the size parameter $\Gamma$, confirming the value of this technique.

\begin{table}[] 
		\fontsize{7.8pt}{7.8pt}\selectfont
		\setlength{\tabcolsep}{4pt}
		\renewcommand{\arraystretch}{1.3}
		\caption{Worst-case cost for the 2736-bus system with and without outer approximations}
		\label{Table:WorstCaseCost:2736BusSystem}
		\centering
		\begin{tabular}{ccccccc}
			\toprule
			$\Gamma$         & 0.25   & 0.5    & 1      & 2      & 3      & 4      \\
			\hline
			With OA (M\$)    & 11.675 & 11.892 & 12.368 & 13.194 & 13.833 & 14.428 \\
			Without OA (M\$) & 11.675 & 11.881 & 12.344 & 13.115 & 13.738 & 14.324 \\
			Difference       & 0.00\% & 0.09\% & 0.19\% & 0.61\% & 0.69\% & 0.72\% \\
			\bottomrule
		\end{tabular}
\end{table}


Given the complexity of the multistage robust UC with dynamic uncertainty sets, the large-scale 2736-bus instance solved here, and the simple computer where these experiments were carried out, we believe that the solution
method proposed here is very promising for an eventual practical implementation in real-world power systems with a significant adoption of wind and solar power.


\subsection{Comparison to other UC and ED models} \label{Section:ComputationalExperiments:Simulation}
This section studies the performance of three different UC solutions and ED methods on a simulation platform of the dispatch process that mimics the hour to hour operation of the power system. This simulation consists of a rolling-horizon process where, given an on/off schedule for generators (UC solution $\mb{x}$), a dispatch problem is solved for every $t=1,\dots,T$, starting with $t=1$ and moving forward until $t=T$, with uncertain parameters at time $t$ revealed only at that time. That is, when solving a dispatch problem at time $t$ the values of uncertain parameters at future time periods are not known. The dispatch problem solved at time $t$ implements dispatch decisions for that time, and it takes as input the dispatch decisions implemented in the previous time periods. $N=100$ such simulations are carried out, with $T=24$ hours, and then several cost and reliability metrics are examined. The trajectories for wind and solar power are generated using the stochastic model in Eq. \eqref{eq:StochasticModel}, using 30 days of data for parameter estimation. For the 2736-bus system, the $N=100$ simulated trajectories present an average of 5164 MW for available wind power and 1133 MW for available solar power, resulting on an average renewable penetration of 35.3\%.

The following UC and ED models are tested: multistage robust UC with dynamic uncertainty sets using the policy-guided look-ahead ED method proposed in section \ref{Section:Model:EDMethods} (RobUC-Dynamic), multistage robust UC with static uncertainty sets using the policy-enforcement ED method proposed in \cite{LorcaSunLitvinovZheng2016AffineUC} (RobUC-Static), and deterministic UC with reserves using deterministic look-ahead ED (DetUC).

The deterministic UC corresponds to a modification of problem \eqref{eq:FullyAdaptiveModel} in the case where the uncertainty set only contains the forecast trajectory for available renewable power, $\overline{\mathcal{P}}^r = \{ \overline{\mb{p}}^{r,forecast}\}$, thus collapsing the dispatch policy to one dispatch plan rather than a function, $\mb{p}_t(\overline{\mb{p}}^r_{[t]}) = \mb{p}_t$. This model is further enhanced by reserves through replacing eq. \eqref{eq:FullyAdaptiveModel:GeneratorOutputLimits} by
\begin{align*}
& x^o_{it} \, \underline{p}^g_{it} + r^-_{it} \leq p^g_{it} \leq x^o_{it} \, \overline{p}^g_{it} - r^+_{it} \quad \forall i \in \mathcal{N}^g, t \in \mathcal{T},
\end{align*}
and adding constraints
\begin{align*}
& \sum_{i \in \mathcal{N}^g} r^-_{it} \geq R^-_t \, , \quad \sum_{i \in \mathcal{N}^g} r^+_{it} \geq R^+_t \quad \forall t \in \mathcal{T},
\end{align*}
where $r^-_{it}, r^+_{it} \geq 0$ are the down-reserve and up-reserve provided by generator $i$ at time $t$, and $R^-_t, R^+_t$ are the down-reserve and up-reserve requirement levels at time $t$. Notice that the multistage robust UC \eqref{eq:FullyAdaptiveModel} does not need to consider reserve requirements, given that it addresses uncertainty in a direct and systematic way, determining reserves endogenously.

The policy-enforcement ED can be formulated as
\begin{align*}
& \min\limits_{\hat{\mb{p}}_t} \; \left\{ \sum_{i \in \mathcal{N}_g} C^g_i \, \hat{p}^g_{it}: \; \mbox{Eqs. \eqref{eq:PolicyGuidedLAED:pt} and \eqref{eq:PolicyGuidedLAED:GeneratorRobustRamping} hold} \right\},
\end{align*}
and the deterministic look-ahead ED as
\begin{align*}
& \min\limits_{\hat{\mb{p}}_t,\dots,\hat{\mb{p}}_{t+T'}} \; \left\{ \sum_{\tau=t}^{t+T'}\sum_{i \in \mathcal{N}_g} C^g_i \, \hat{p}^g_{i\tau}: \; \mbox{Eqs. \eqref{eq:PolicyGuidedLAED:pt} and \eqref{eq:PolicyGuidedLAED:ptau} hold} \right\}.
\end{align*}
With this we can see that the policy-guided look-ahead ED \eqref{eq:PolicyGuidedLAED} generalizes both the above EDs by utilizing the affine policies obtained from the multistage robust UC in the robust ramping constraints \eqref{eq:PolicyGuidedLAED:GeneratorRobustRamping} and enforcing storage levels constraints \eqref{eq:PolicyGuidedLAED:StorageCapacityDeterminedByPolicy} in a multi-period look-ahead framework, which are both very important to leverage the benefits of energy storage resources. Here, we use $T'=3$ look-ahead periods.

For the two robust UC models, the size of the uncertainty sets is parameterized by $\Gamma$, and for the deterministic UC, the reserve requirement levels $R^-_t, R^+_t$ are selected as $R^-_t = R^+_t \; = \; \Gamma \; \sigma^{TNL}_t$, where $\sigma^{TNL}_t$ is the standard deviation of total net load (namely, total demand minus total available renewable power) at time $t$, under the simulated trajectories \cite{BlackStrbac2007StdDevForReserve}. To properly study the performance of all these methods, ED problems are extended with penalty variables for violations of energy balance and transmission line capacity limits, each with a unit cost of \$5000/MWh.

The simulation results for the 2736-bus system are presented in Table \ref{Table:SimulationResults:2736BusSystem}, where ``Cost Avg'' is the average of total cost over the $N=100$ simulations, ``Cost Std'' is the standard deviation of total cost, ``Cost CVaR'' is the conditional value at risk of total cost at a 10\% level (that is, the average total cost of the 10 highest total costs, given $N=100$), ``Penalty Cost Avg'' is the average penalty cost, ``Penalty Freq'' is the proportion of time periods where penalty occurred, ``Renewables Util'', utilization of renewables, is the proportion of used renewable power with respect to available renewable power, and ``Stored Avg'' is the average level of stored energy.


\begin{table}[] 
  \fontsize{7.8pt}{7.8pt}\selectfont
  \setlength{\tabcolsep}{3.5pt}
  \renewcommand{\arraystretch}{1.3}
  \caption{Simulation results for Polish 2736-bus system}
  \label{Table:SimulationResults:2736BusSystem}
  \centering
  \begin{tabular}{ccccccc}
  \toprule
  \multicolumn{7}{c}{Multistage robust UC with dynamic uncertainty set} \\
  \multicolumn{7}{c}{using policy-guided look-ahead ED (RobUC-Dynamic)} \\
    \hline
    $\Gamma$          & 0.25   & 0.5    & 1      & 2      & 3      & 4       \\
    \hline
    Cost Avg (M\$)    & 12.089 & 11.459 & 11.567 & 11.729 & 11.865 & 12.017  \\
    Cost Std (M\$)    & 1.991  & 0.262  & 0.189  & 0.199  & 0.202  & 0.200   \\
    Cost CVaR (M\$)   & 17.343 & 12.000 & 11.907 & 12.086 & 12.228 & 12.377  \\
    Penalty Cost Avg (\$) & 29424  & 884    & 0      & 0      & 0      & 0       \\
    Penalty Freq      & 4.67\% & 0.29\% & 0.00\% & 0.00\% & 0.00\% & 0.00\%  \\
    Renewables Util    & 99.2\% & 99.1\% & 99.1\% & 98.6\% & 97.6\% & 96.5\% \\
    Stored Avg (MWh)  & 613    & 689    & 726    & 709    & 847    & 1080    \\
    \hline
  \end{tabular}
  \begin{tabular}{ccccccccc}
  \multicolumn{7}{c}{} \\
  \multicolumn{7}{c}{Multistage robust UC with static uncertainty set} \\
  \multicolumn{7}{c}{using policy-enforcement ED (RobUC-Static)} \\
    \hline
    $\Gamma$          & 0.25   & 0.5    & 1      & 2      & 3      & 4       \\
    \hline
    Cost Avg (M\$)    & 24.844  & 18.215  & 13.676 & 11.671 & 11.639 & 11.765 \\
    Cost Std (M\$)    & 11.654  & 7.998   & 3.851  & 0.608  & 0.209  & 0.213  \\
    Cost CVaR (M\$)   & 49.390  & 36.601  & 23.397 & 13.037 & 12.023 & 12.156 \\
    Penalty Cost Avg (\$) & 563688  & 285140  & 93386  & 4696   & 0      & 0      \\
    Penalty Freq      & 31.63\% & 19.83\% & 8.63\% & 0.58\% & 0.00\% & 0.00\% \\
    Renewables Util    & 98.6\% & 98.6\%  & 98.3\% & 97.5\% & 96.7\% & 96.1\% \\
    Stored Avg (MWh)  & 26      & 46      & 88     & 101    & 129    & 151    \\
    \hline
  \end{tabular}
  \begin{tabular}{ccccccccc}
  \multicolumn{7}{c}{} \\
  \multicolumn{7}{c}{Deterministic UC with reserve using deterministic look-ahead ED (DetUC)} \\
    \hline
    $\Gamma$          & 0.25   & 0.5    & 1      & 2      & 3      & 4       \\
    \hline
    Cost Avg (M\$)        & 18.405  & 20.878  & 16.687  & 16.485  & 17.018  & 12.520 \\
    Cost Std (M\$)        & 6.935   & 7.446   & 5.605   & 5.363   & 5.805   & 2.260  \\
    Cost CVaR (M\$)       & 34.362  & 36.326  & 30.015  & 29.058  & 29.739  & 18.458 \\
    Penalty Cost Avg (\$) & 295995  & 400152  & 223927  & 215491  & 237765  & 48136  \\
    Penalty Freq          & 37.13\% & 32.83\% & 32.33\% & 32.08\% & 27.92\% & 9.63\% \\
    Renewables Util       & 98.8\%  & 99.0\%  & 98.8\%  & 98.8\%  & 98.9\%  & 99.0\% \\
    Stored Avg (MWh)      & 384     & 380     & 381     & 381     & 369     & 377    \\
  \bottomrule
  \end{tabular}
\end{table}


\subsubsection{Robust UC v.s. Deterministic UC}
\underline{First}, for both robust and deterministic UC models, the trade-off between operational cost and system reliability (cost std, penalty freq, CVaR) is controlled by the uncertainty set size parameter $\Gamma$ and the reserve level parameter, respectively. Higher $\Gamma$ or reserve level improves system reliability but may increase cost. \underline{Second}, the robust UC models significantly improve both the operational cost and reliability over DetUC. In particular, comparing to the best economic performance of DetUC ($\Gamma=4$), the RobUC-Dynamic model at $\Gamma=1$ achieves a decrease of $7.62\%$ in average cost, $91.64\%$ in cost std, $35.49\%$ in CVaR, and at the same time completely eliminates penalty; the RobUC-Static model at $\Gamma=3$ achieves a decrease of $7.04\%$ in average cost, $90.77\%$ in cost std, and $34.87\%$ in CVaR, and also eliminates penalty. Notice that, at its best performance ($\Gamma=4$), DetUC still has substantial penalty cost. See Figure \ref{Figure:BarPlots} for a graphical representation of this comparison. \underline{Third}, the robust models curtail renewables slightly more than the DetUC model to achieve significantly improved system reliability.

\begin{figure}[!t]
\centering
\includegraphics[scale=0.30]{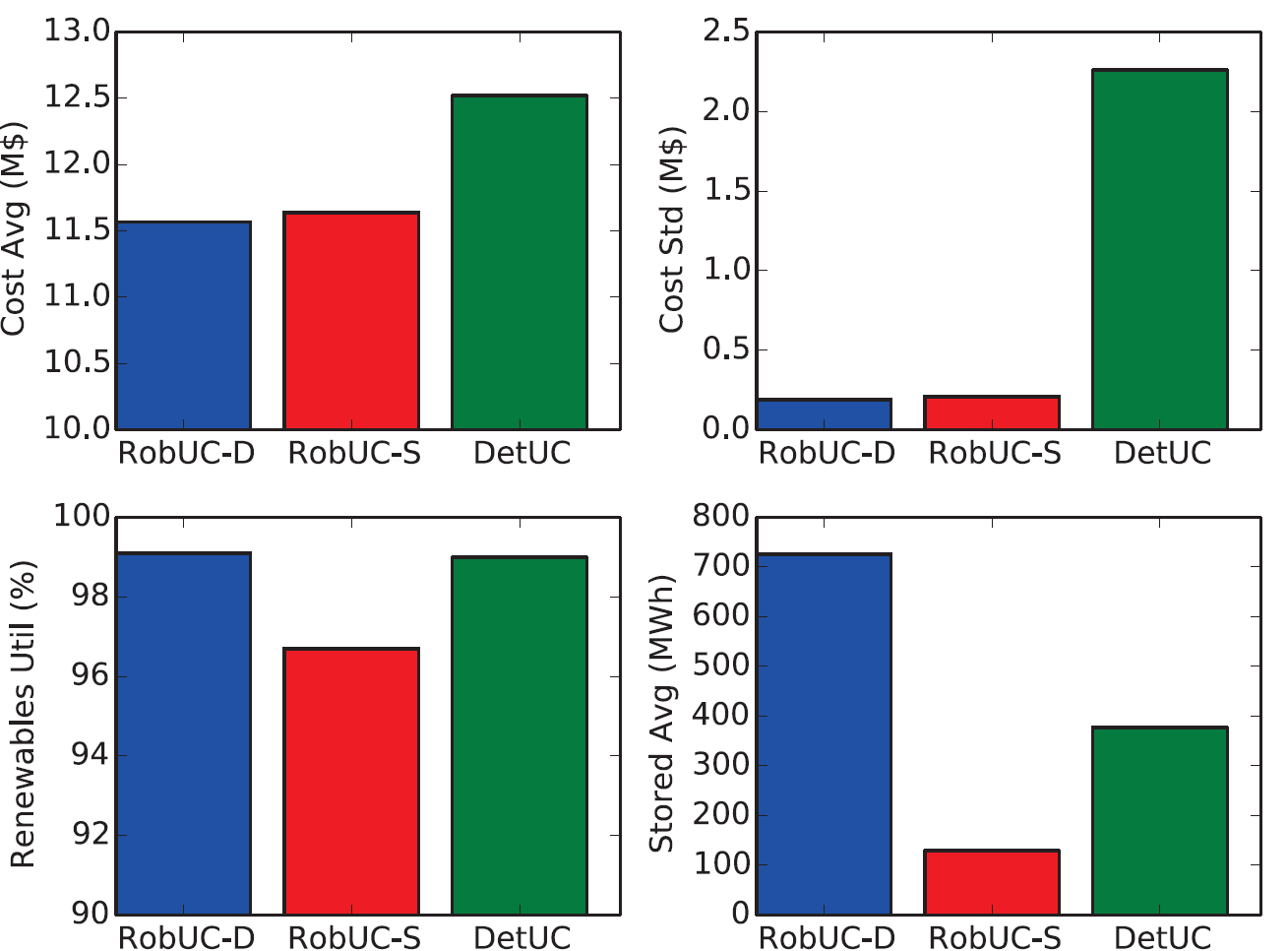}
\caption{Performance measures for RobUC-Dynamic at $\Gamma=1$, RobUC-Static at $\Gamma=3$ and DetUC at $\Gamma=4$}
\label{Figure:BarPlots}
\end{figure}


\subsubsection{RobUC-Dynamic v.s. RobUC-Static}
Robust UC with dynamic uncertainty sets (RobUC-Dynamic) further improves over robust UC with static uncertainty sets (RobUC-Static). \underline{First}, RobUC-Dynamic at $\Gamma=1$ achieves better performance in all three categories: $0.62\%$ lower cost avg, $9.57\%$ lower cost std, and $0.96\%$ lower CVaR than RobUC-Static at its best performance $\Gamma=3$. In other words, robust UC with dynamic uncertainty sets \emph{dominates} the performance of static uncertainty sets. \underline{Second}, we can also observe that the average level of storage utilization for RobUC-Dynamic is much higher than that of RobUC-Static. This difference is mainly driven by the different ED models used. In particular, the policy guided look-ahead ED in RobUC-Dynamic is more effective at deploying energy storage devices than the policy-enforcement ED by following the multistage policy for storage decisions. \underline{Third}, RobUC-Dynamic utilizes more renewable power than RobUC-Static for all levels of $\Gamma$. This is due to the fact that the dynamic uncertainty set is more realistic and less conservative than the static uncertainty set.



In summary, robust UC models dominate the deterministic UC model in all operational cost and system reliability metrics. Moreover, the multistage robust UC model with dynamic uncertainty sets and policy-guided look-ahead ED dominates RobUC-Static. RobUC-Dynamic also exhibits higher utilization of storage devices and reduces renewable curtailment.

\section{Conclusion} \label{Section:Conclusion}
We present a multistage robust UC model with dynamic uncertainty sets for power systems with significant wind and solar power and storage units. We also propose a novel dispatch process to accompany the robust UC model. An efficient solution framework based on constraint generation and duality reformulations, with several algorithmic improvements, is developed. With extensive computational experiments, we show that the proposed algorithm can solve large-scale multistage robust UC models with high dimensional uncertainty in a time budget suitable for the day-ahead operation. The proposed robust UC model with the novel ED method is shown to dominate the deterministic UC with reserve and look-ahead ED in both operational cost and system reliability. The proposed dynamic uncertainty sets also effectively capture the temporal and spatial correlations of wind and solar power, which is important for further improving the performance of the multistage robust UC model. The new ED method leads to more utilization of storage units and less curtailment of renewable power.


In summary, the proposed multistage robust UC model, the dynamic uncertainty sets, the policy-guided look-ahead ED, and the solution methodology significantly improve over the existing deterministic and multistage robust UC models and solution methods, and provide a novel and effective approach for operating large-scale power systems with a large number of wind and solar farms and storage devices. Finally, a challenging topic for future work is to incorporate security constraints into the multistage robust UC framework.



\bibliographystyle{IEEEtran} 
\bibliography{RobUCDUSBib}

\begin{thebibliography}{10}
\providecommand{\url}[1]{#1}
\csname url@samestyle\endcsname
\providecommand{\newblock}{\relax}
\providecommand{\bibinfo}[2]{#2}
\providecommand{\BIBentrySTDinterwordspacing}{\spaceskip=0pt\relax}
\providecommand{\BIBentryALTinterwordstretchfactor}{4}
\providecommand{\BIBentryALTinterwordspacing}{\spaceskip=\fontdimen2\font plus
\BIBentryALTinterwordstretchfactor\fontdimen3\font minus
  \fontdimen4\font\relax}
\providecommand{\BIBforeignlanguage}[2]{{%
\expandafter\ifx\csname l@#1\endcsname\relax
\typeout{** WARNING: IEEEtran.bst: No hyphenation pattern has been}%
\typeout{** loaded for the language `#1'. Using the pattern for}%
\typeout{** the default language instead.}%
\else
\language=\csname l@#1\endcsname
\fi
#2}}
\providecommand{\BIBdecl}{\relax}
\BIBdecl

\bibitem{CheungWatson2015StochasticUC}
K.~Cheung, D.~Gade, C.~Silva-Monroy, S.~M. Ryan, J.~P. Watson, R.~J. Wets, and
  D.~L. Woodruff, ``Toward scalable stochastic unit commitment: Part 2: solver
  configuration and performance assessment,'' \emph{Energy Systems}, 2015.

\bibitem{PapavasiliouOren2015HighPerformance}
A.~Papavasiliou, S.~S. Oren, and B.~Rountree, ``Applying high performance
  computing to transmission-constrained stochastic unit commitment for
  renewable energy integration,'' \emph{IEEE Transactions on Power Systems},
  vol.~30, no.~3, pp. 1109--1120, 2015.

\bibitem{Takriti1996Stochastic}
S.~Takriti, J.~R. Birge, and E.~Long, ``A stochastic model for the unit
  commitment problem,'' \emph{IEEE Transactions on Power Systems}, vol.~11,
  no.~3, pp. 1497--1508, 1996.

\bibitem{WangShahidehpour2008}
J.~Wang, M.~Shahidehpour, and Z.~Li, ``Security-constrained unit commitment
  with volatile wind power generation,'' \emph{IEEE Transactions on Power
  Systems}, vol.~23, no.~3, pp. 1319--1327, 2008.

\bibitem{Street2011}
A.~Street, F.~Oliveira, and J.~M. Arroyo, ``Contingency-constrained unit
  commitment with $n-k$ security criterion: A robust optimization approach,''
  \emph{IEEE Transactions on Power Systems}, vol.~26, no.~3, pp. 1581--1590,
  2011.

\bibitem{BertsimasSun2013Adaptive}
D.~Bertsimas, E.~Litvinov, X.~A. Sun, J.~Zhao, and T.~Zheng, ``Adaptive robust
  optimization for the security constrained unit commitment problem,''
  \emph{IEEE Transactions on Power Systems}, vol.~28, no.~1, pp. 52--63, 2013.

\bibitem{Jiang2012}
R.~Jiang, J.~Wang, and Y.~Guan, ``Robust unit commitment with wind power and
  pumped storage hydro,'' \emph{IEEE Transactions on Power Systems}, vol.~27,
  no.~2, pp. 800--810, 2012.

\bibitem{Zeng2012robustUC}
L.~Zhao and B.~Zeng, ``Robust unit commitment problem with demand response and
  wind energy,'' in \emph{IEEE Power and Energy Society General Meeting}, 2012.

\bibitem{WangWatson2013}
Q.~Wang, J.~P. Watson, and Y.~Guan, ``Two-stage robust optimization for $n-k$
  contingency-constrained unit commitment,'' \emph{IEEE Transactions on Power
  Systems}, vol.~28, no.~3, pp. 2366--2375, 2013.

\bibitem{ZhaoGuan2013}
C.~Zhao and Y.~Guan, ``Unified stochastic and robust unit commitment,''
  \emph{IEEE Transactions on Power Systems}, vol.~28, no.~3, pp. 3353--3361,
  2013.

\bibitem{Jiang2013MinimaxRegret}
R.~Jiang, J.~Wang, M.~Zhang, and Y.~Guan, ``Two-stage minimax regret robust
  unit commitment,'' \emph{IEEE Transactions on Power Systems}, vol.~28, no.~3,
  pp. 2271--2282, 2013.

\bibitem{Jiang2014TwoStage}
R.~Jiang, M.~Zhang, G.~Li, and Y.~Guan, ``Two-stage network constrained robust
  unit commitment problem,'' \emph{European Journal of Operational Research},
  vol. 234, no.~3, pp. 751--762, 2014.

\bibitem{SunLorca2014}
X.~A. Sun and A.~Lorca, ``Adaptive robust optimization for daily power system
  operation,'' in \emph{Power Systems Computation Conference (PSCC)}.\hskip 1em
  plus 0.5em minus 0.4em\relax IEEE, 2014, pp. 1--9.

\bibitem{LorcaSunLitvinovZheng2016AffineUC}
A.~Lorca, X.~A. Sun, E.~Litvinov, and T.~Zheng, ``Multistage adaptive robust
  optimization for the unit commitment problem,'' \emph{Operations Research},
  vol.~64, no.~1, pp. 32--51, 2016.

\bibitem{BenTal2004}
A.~Ben-Tal, A.~Goryashko, E.~Guslitzer, and A.~Nemirovski, ``Adjustable robust
  solutions of uncertain linear programs,'' \emph{Mathematical programming,
  Ser. A}, vol.~99, pp. 351--376, 2004.

\bibitem{Warrington2015RollingUC}
J.~Warrington, C.~Hohl, P.~J. Goulart, and M.~Morari, ``Rolling unit commitment
  and dispatch with multi-stage recourse policies for heterogeneous devices,''
  \emph{IEEE Transactions on Power Systems}, 2015.

\bibitem{Warrington2013PolicyBasedReserves}
J.~Warrington, P.~J. Goulart, S.~Mari\'ethoz, and M.~Morari, ``Policy-based
  reserves for power systems,'' \emph{IEEE Transactions on Power Systems},
  vol.~28, no.~4, pp. 4427--4437, 2013.

\bibitem{Jabr2013AdjustableOPF}
R.~A. Jabr, ``Adjustable robust {OPF} with renewable energy sources,''
  \emph{IEEE Transactions on Power Systems}, vol.~28, no.~4, pp. 4742--4751,
  2013.

\bibitem{Bienstock2014ChanceConstrained}
D.~Bienstock, M.~Chertkov, and S.~Harnett, ``Chance-constrained optimal power
  flow: risk-aware network control under uncertainty,'' \emph{SIAM Review},
  vol.~56, no.~3, pp. 461--495, 2014.

\bibitem{Delage2015Tutorial}
E.~Delage and D.~Iancu, ``Robust multi-stage decision making,'' \emph{INFORMS
  Tutorials in Operations Research}, pp. 20--46, 2015.

\bibitem{Chen2007ARobust}
X.~Chen, M.~Sim, and P.~Sun, ``A robust optimization perspective on stochastic
  programming,'' \emph{Operations Research}, vol.~55, no.~6, pp. 1058--1071,
  2007.

\bibitem{Minoux2014TwoStage}
M.~Minoux, ``Two-stage robust optimization, state-space representable
  uncertainty and applications,'' \emph{RAIRO-Operations Research}, vol.~48,
  no.~4, pp. 455--475, 2014.

\bibitem{LorcaSun2015Wind}
A.~Lorca and X.~A. Sun, ``Adaptive robust optimization with dynamic uncertainty
  sets for multi-period economic dispatch under significant wind,'' \emph{IEEE
  Transactions on Power Systems}, vol.~30, no.~4, pp. 1702--1713, 2015.

\bibitem{Shapiro2009StochBook}
A.~Shapiro, D.~Dentcheva, and A.~P. Ruszczy{\'n}ski, \emph{Lectures on
  stochastic programming: modeling and theory}.\hskip 1em plus 0.5em minus
  0.4em\relax SIAM, 2009, vol.~9.

\bibitem{Ostrowski2012TightUC}
J.~Ostrowski, M.~F. Anjos, and A.~Vannelli, ``Tight mixed integer linear
  programming formulations for the unit commitment problem,'' \emph{IEEE
  Transactions on Power Systems}, vol.~27, no.~1, p.~39, 2012.

\bibitem{Kuhn2011primal}
D.~Kuhn, W.~Wiesemann, and A.~Georghiou, ``Primal and dual linear decision
  rules in stochastic and robust optimization,'' \emph{Mathematical
  Programming}, vol. 130, no.~1, pp. 177--209, 2011.

\bibitem{Xie2011WindIntegration}
L.~Xie, P.~Carvalho, L.~Ferreira, J.~Liu, B.~Krogh, N.~Popli, and M.~Ilic,
  ``Wind integration in power systems: Operational challenges and possible
  solutions,'' \emph{Proceedings of the IEEE}, vol.~99, no.~1, pp. 214--232,
  2011.

\bibitem{Reinsel2003Elements}
G.~C. Reinsel, \emph{Elements of multivariate time series analysis}.\hskip 1em
  plus 0.5em minus 0.4em\relax Springer Science \& Business Media, 2003.

\bibitem{Bental2009robustbook}
A.~Ben-Tal, L.~El~Ghaoui, and A.~Nemirovski, \emph{Robust optimization}.\hskip
  1em plus 0.5em minus 0.4em\relax Princeton University Press, 2009.

\bibitem{BertsimasDunning2015Reformulations}
D.~Bertsimas, I.~Dunning, and M.~Lubin, ``Reformulations versus cutting planes
  for robust optimization,'' \emph{Computational Management Science}, 2015.

\bibitem{Tuybook}
H.~Tuy, \emph{Convex analysis and global optimization}.\hskip 1em plus 0.5em
  minus 0.4em\relax Kluwer Academic Publishers, 1998.

\bibitem{Cplex2009}
{IBM ILOG CPLEX}, ``{V12. 1: User's Manual for CPLEX},'' \emph{International
  Business Machines Corporation}, 2009.

\bibitem{Zimmerman11MATPOWER}
R.~D. Zimmerman, C.~E. Murillo-S\'anchez, and R.~J. Thomas, ``{MATPOWER}:
  Steady-state operation, planning and analysis tools for power systems
  research and education,'' \emph{IEEE Transactions on Power Systems}, vol.~26,
  no.~1, pp. 12--19, 2011.

\bibitem{Denholm2013ValueOfStorage}
P.~Denholm, J.~Jorgenson, M.~Hummon, T.~Jenkin, D.~Palchak, B.~Kirby, O.~Ma,
  and M.~O'Malley, ``The value of energy storage for grid applications,''
  \url{http://www.nrel.gov/docs/fy13osti/58465.pdf}, 2013, {NREL Technical
  Report}.

\bibitem{Potter2008CreatingDataset}
C.~W. Potter, D.~Lew, J.~McCaa, S.~Cheng, S.~Eichelberger, and E.~Grimit,
  ``Creating the dataset for the western wind and solar integration study
  ({USA}),'' \emph{Wind Engineering}, vol.~32, no.~4, pp. 325--338, 2008.

\bibitem{BlackStrbac2007StdDevForReserve}
M.~Black and G.~Strbac, ``Value of bulk energy storage for managing wind power
  fluctuations,'' \emph{IEEE Transactions on Energy Conversion}, vol.~22,
  no.~1, pp. 197--205, 2007.

\end{thebibliography}

\end{document}